\documentclass[12pt]{amsart}
\usepackage{amsmath}
\usepackage{amssymb}
\usepackage[mathcal]{eucal}
\usepackage[all]{xy}
\usepackage{latexsym}
\usepackage{amstext}
\usepackage{amsfonts}
\usepackage{amsthm}
\usepackage{amsopn}
\usepackage{amsbsy}
\usepackage{layout}
\usepackage{color}
\usepackage{graphicx}
\usepackage{bm}

\newcommand{\supp}{\mathrm{supp}}

\theoremstyle{plain}
\newtheorem{theorem}{Theorem}
\newtheorem{proposition}[theorem]{Proposition}

\newtheorem{corollary}[theorem]{Corollary}

\theoremstyle{definition}

\theoremstyle{remark}

\newtheorem{remark}[theorem]{Remark}
\newtheorem{example}[theorem]{Example}

\def\N{{\rm I}\hskip-.13em{\rm N}}
\def\R{{\rm I}\hskip-.13em{\rm R}}

\date{}

\begin{document}

\title[An elementary proof of the Josefson-Nissenzweig theorem]{An elementary proof of the Josefson-Nissenzweig theorem for Banach spaces $C(K\times L)$}
\author{Jerzy K{\c{a}}kol}
\address{Faculty of Mathematics and Informatics. A. Mickiewicz University,
61-614 Pozna\'{n}}
\email{kakol@amu.edu.pl}
\author{Wies{\l}aw \'Sliwa}
\address{Faculty of Exact and Technical Sciences,
University of Rzesz\'ow%
\newline
\indent%
35-310 Rzesz\'ow, Poland }
\email{sliwa@amu.edu.pl; wsliwa@ur.edu.pl}



\begin{abstract}
In \cite{Kakol-Sobota-Zdomskyy} probabilistic methods, in particular a variant of the Weak Law of Large Numbers related to the Bernoulli distribution, have been used to show that for every infinite compact spaces $K$ and $L$ there exists   a sequence  $(\mu_n)$ of normalized signed measures on $K\times L$ with finite supports which converges to $0$ with respect to the weak$^*$ topology of the dual Banach space $C(K\times L)^{*}$.

In this paper, we return to this construction, limiting ourselves only to elementary combinatorial calculus. The main effects of this construction are additional information about the measures $\mu_n$, this is particularly clearly seen (among the others) in the resulting inequalities
$$\frac{1}{2\sqrt{\pi}}\frac{1}{\sqrt{n}} <\sup_{A\times B\subset X\times Y} |\mu_n(A\times B)|<\frac{2}{\sqrt{\pi}}\frac{1}{\sqrt{n}},$$
 $n\in\mathbb{N}$, with $\mu_n(f) \to_n 0$ for every $f\in C(X \times Y),$ where $X$ and $Y$ are arbitrary Tychonoff spaces containing infinite compact subsets, respectively.
As an application we explicitly describe for Banach spaces $C(X\times Y)$  some complemented subspaces isomorphic to $c_0$.
This result  generalizes the classical theorem of Cembranos and Freniche, which states that for every infinite compact spaces $K$ and $L$, the Banach space $C(K\times L)$ contains a complemented copy of the Banach space $c_0$.
\end{abstract}
\maketitle


\maketitle


\section{Introduction and motivations}

\bigskip

Let $X$ be  a Tychonoff space. By $C_{p}(X)$ we denote the space of real-valued continuous functions on $X$ endowed with the pointwise topology. For a compact (Hausdorff) space $X$  by $C(X)$ we denote the Banach space of continuous real-valued functions on $X$ and $c_0$ denotes the Banach space of all real-valued sequences which converge to $0$, both equipped with the uniform norm topology.

It is well known that  for every infinite compact space $X$ the Banach space $C(X)$ contains a copy of $c_0$ but there exist several examples od compact $X$ for which  $C(X)$ does not contain a complemented copy of $c_0$.   However, if $X$ is a product of two infinite compact spaces,  $C(X)$ always contains a complemented copy of $c_0$, as was proved by Cembranos \cite{Ce} and Freniche \cite{Fre84}.

An essential  step  of the proof of this result  was  an application of the classical Josefson--Nissenzweig theorem for Banach spaces $C(X)$ stating that for every infinite compact space $X$  there is a sequence $(\mu_n)$ of normalized signed regular Borel measures on $X$ which converges to $0$ with respect to the weak* topology of the dual space $C(X)^*$.

In \cite[Theorem 1.2]{Kakol-Sobota-Zdomskyy} we used some  tools from probability theory, in particular a variant of the Weak Law of Large
 Numbers related to the Bernoulli distribution to show   that for every infinite compact spaces $K$ and $L$ there exists   a sequence  $(\mu_n)$ of normalized signed measures on $K\times L$ with finite supports which converges to $0$ with respect to the weak$^*$ topology of the dual Banach space $C(K\times L)^{*}$.

 The main goal (Theorem \ref{MAIN}) of our present work is, on the one hand, to obtain much sharper general properties of the presented sequence of measures $(\mu_n)$ (defined similarly to \cite{Kakol-Sobota-Zdomskyy}) and, on the other hand, to use only very  elementary combinatorial calculus in their proof.

  A particularly striking property of the sequence $(\mu_n)$ that we show is the equality \[ \sup_{A\times B\subset X\times Y} |\mu_n(A\times B)|=\frac{1}{n2^n}\left\lceil \frac{n}{2} \right\rceil\binom{n}{\lceil \frac{n}{2} \rceil} \; \mbox{for every}\; n\in \mathbb{N};\] it follows
   the inequality  \[\frac{1}{2\sqrt{\pi}}\frac{1}{\sqrt{n}} <\sup_{A\times B\subset X\times Y} |\mu_n(A\times B)|<\frac{2}{\sqrt{\pi}}\frac{1}{\sqrt{n}} \; \mbox{for every}\; n\in \mathbb{N},\]
  with $\mu_n(f) \to_n 0$ for every $f\in C(X \times Y),$ where  $X$ and $Y$ are Tychonoff spaces containing infinite compact subsets and $\lceil p\rceil$ denotes the least integer that is greater or equal to the real number $p$.

As an application of  Theorem \ref{MAIN} we explicitly describe  for Banach spaces $C(X\times Y)$  some complemented subspaces isomorphic to $c_0$, see Corollary \ref{LAST2}.

Note that  all standard proofs of the Josefson--Nissenzweig theorem (see for example \cite{Diestel}) were  presented as   non-constructive, it was  difficult   to deduce from the proofs of Cembranos and Freniche how the constructed complemented copy of $c_0$ in a given space $C(K\times L)$ looks like.

For a Tychonoff space $X$ and a point $x\in X$ let $\delta_x:C_p(X)\to\mathbb{R},\,\,\, \delta_x:f\mapsto f(x),$ be the Dirac measure concentrated at $x$. The linear hull $L_p(X)$ of the set $\{\delta_x:x\in X\}$ in $\mathbb{R}^{C_p(X)}$ can be identified with the dual space of $C_p(X)$.
Elements of the space $L_p(X)$ will be called {\em finitely supported sign-measures} (or simply {\em sign-measures}) on $X$.

Each non-zero $\mu\in L_p(X)$ can be uniquely written as a linear combination of Dirac measures $\mu=\sum_{x\in F}\alpha_x\delta_x$ for some finite set $F\subset X$ and some non-zero real numbers $\alpha_x$. The set $F$ is called the {\em support} of the sign-measure $\mu$ and is denoted by $\supp(\mu)$. The measure $\sum_{x\in F}|\alpha_x|\delta_x$ will be denoted by $|\mu|$ and the real number $\|\mu\|=\sum_{x\in F}|\alpha_x|$ coincides with the {\em norm} of $\mu$ (in the dual Banach space $C(\beta X)^*$).

Following \cite{BKS1} we say
that for  a   Tychonoff space $X$  the space $C_{p}(X)$ satisfies the \emph{Josefson-Nissenzweig property} (JNP in short) if there exists a sequence $(\mu_n)$ of finitely supported sign-measures  on $X$ such that $\|\mu_n\|=1$ for all $n\in \mathbb{N}$, and $\mu_n(f)\to_n 0$ for each $f\in C(X)$. $\\$ The corresponding sequence $(\mu_n)$ of signed measures  is also called a JN-sequence. The existence of JN-sequences  and their properties is currently being intensively researched; see, for example, recent papers \cite{MSZ} and \cite{MS}, \cite{KMSZ}, \cite{Kakol-Sobota-Zdomskyy}.

\bigskip

Our main result is the following:
\begin{theorem}\label{MAIN}
Let $X$ and $Y$ be Tychonoff spaces that contain infinite compact subspaces $K$ and $L$, respectively. Let $n\in \N.$ Let $K_n\times L_n$ be a finite subset of $K\times L$ such that $|K_n|=2^n$ and $|L_n|=n$. Let $\varphi_n: K_n \to \{-1, 1\}^{L_n}$ be a bijection. Then \[\mu_n:= \frac{1}{n2^n}\sum_{(s,j)\in K_n\times L_n} \varphi_n(s)(j)\delta_{(s,j)}\] is a finitely supported signed measure on $X\times Y$ such that

\[(1)\;\;\; \|\mu_n\|=1 \; \mbox{and}\; \mbox{supp}\, (\mu_n)=K_n \times L_n;\]
\[(2) \sup_{A\times B\subset X\times Y} |\mu_n(A\times B)|=\frac{1}{n2^n}\sum_{i=\lceil \frac{n}{2}\rceil}^n (2i-n)\binom{n}{i} =\frac{1}{n2^n}\left\lceil \frac{n}{2} \right\rceil\binom{n}{\lceil \frac{n}{2} \rceil};\]
\[(3)\;\;\; \mu_n(K_{(n)}\times L_n)= \frac{1}{n2^n}\left\lceil \frac{n}{2} \right\rceil\binom{n}{\lceil \frac{n}{2} \rceil} \; \mbox{for some}\; K_{(n)}\subset K_n;\]
\[(4)\;\;\; \frac{1}{2\sqrt{\pi}}\frac{1}{\sqrt{n}} <\sup_{A\times B\subset X\times Y} |\mu_n(A\times B)|<\frac{2}{\sqrt{\pi}}\frac{1}{\sqrt{n}}.\]
\[(5)\;\;\; |\mu_n(f\otimes g)|\leq \frac{8}{\sqrt{\pi}}\frac{1}{\sqrt{n}}\|f\otimes g\|_{\infty}\;\mbox{and}\; |\mu_n(f\oplus g)|\leq \frac{8}{\sqrt{\pi}}\frac{1}{\sqrt{n}}\|f\oplus g\|_{\infty}\]
for all $f\in C(K)$ and $g\in C(L)$, where $f\otimes g, f\oplus g: K\times L \to \R$ with $(f\otimes g) (x,y) =f(x)g(y)$ and $ (f\oplus g) (x,y) = f(x)+g(y).$
It follows that $\mu_n(f) \to_n 0$ for every $f\in C(X \times Y).$ Thus $(\mu_n)$ is a JN-sequence.
\end{theorem}
\begin{remark}
(1) Let us note that the surprising property (4) of the above Theorem \ref{MAIN} is the result of the construction we present, which is completely different from the one shown in the mentioned theorem from  \cite[Theorem 1.2]{Kakol-Sobota-Zdomskyy}.
(2) Note that the assumptions on spaces $X$ and $Y$ in Theorem \ref{MAIN} cannot be omitted, see Example \ref{example1} and Example \ref{example2} below.
\end{remark}

In \cite[Theorem 1.1]{BKS1} we proved the following  fundamental
\begin{theorem}[Banakh--K\c akol--\'Sliwa]\label{BKS}
For  a Tychonoff space $X$ the space  $C_p(X)$  has JNP if and only if $C_p(X)$ contains a complemented  copy of the space $(c_0)_p$, i.e. the space $c_0$ endowed with the topology inherited from the product $\mathbb{R}^\mathbb{N}$.
\end{theorem}
This combined with our Theorem \ref{MAIN} applies immediately to get \cite[Corollary 2.6]{KMSZ} stating that whenever $X$ and $Y$ contain infinite compact subsets, then $C_p(X\times Y)$ has a complemented copy of $(c_0)_p$.
We prove a stronger fact, see  Corollary \ref{LAST1}. Moreover, if $X$  and $Y$ are  locally compact and $\sigma$-compact spaces, the metrizable and complete lcs  $C_k(X\times Y)$   contains either a complemented copy of $\mathbb{R}^{\mathbb{N}}$ or a complemented copy of the Banach space $c_0$, see Corollary \ref{BKS}.

\bigskip

\section{Proof of Theorem \ref{MAIN} and corollaries}
\bigskip

\begin{proof}

Let $n\in \N.$ Clearly $\|\mu_n\|=1$ and $\mbox{supp} (\mu_n)=K_n \times L_n$.

Let $\emptyset \neq A \subset K_n, \emptyset \neq B \subset L_n$ and $k=|B|.$ For $0\leq i \leq k$ let $A_i$ be the set of all $s\in A$ such that $$|\{j\in B:  \varphi_n(s)(j)=1\}|=i.$$ Clearly $|A_i|\leq \binom{k}{i}2^{n-k};$ if $A=K_n$, then $|A_i|=2^{n-k}\binom{k}{i}.$ We have
\[ \mu_n(A\times B)=\frac{1}{n2^n}\sum_{(s,j)\in A\times B}\varphi_n(s)(j)= \frac{1}{n2^n}\sum_{i=0}^k \sum_{s\in A_i} \sum_{j\in B} \varphi_n(s)(j)=\] \[ \frac{1}{n2^n}\sum_{i=0}^k \sum_{s\in A_i} [i-(k-i)]=\frac{1}{n2^n}\sum_{i=0}^k \sum_{s\in A_i} (2i-k)\leq \frac{1}{n2^n}\sum_{i=\lceil \frac{k}{2}\rceil}^k \sum_{s\in A_i} (2i-k)=\] \[ \frac{1}{n2^n}\sum_{i=\lceil \frac{k}{2}\rceil}^k (2i-k)|A_i|\leq \frac{1}{n2^n}\sum_{i=\lceil \frac{k}{2}\rceil}^k (2i-k)\binom{k}{i}2^{n-k}=\frac{1}{n2^k}\sum_{i=\lceil \frac{k}{2}\rceil}^k (2i-k)\binom{k}{i}.\]
Since \[\sum_{i=\lceil \frac{k}{2}\rceil}^k \binom{k}{i}=\sum_{i=\lceil \frac{k}{2}\rceil}^k \binom{k}{k-i}=\sum_{j=0}^{[ \frac{k}{2}]} \binom{k}{j}\; \mbox{and}\; \sum_{i=0}^k \binom{k}{i}=2^k\] we get $\sum_{i=\lceil \frac{k}{2}\rceil}^k \binom{k}{i}=2^{k-1}$ if $k$ is odd and $\sum_{i=\lceil \frac{k}{2}\rceil}^k \binom{k}{i}=2^{k-1}+\frac{1}{2}\binom{k}{k/2}$ if $k$ is even.

Since $i\binom{k}{i}=k\binom{k-1}{i-1}$ for $1\leq i \leq k$, we obtain \[S_k:=\sum_{i=\lceil \frac{k}{2}\rceil}^k (2i-k)\binom{k}{i}=\left\lceil\frac{k}{2}\right\rceil\binom{k}{\lceil\frac{k}{2}\rceil}.\]
Indeed, if $k$ is odd, then
\[S_k=2\sum_{i=\lceil \frac{k}{2}\rceil}^k i\binom{k}{i}- k\sum_{i=\lceil \frac{k}{2}\rceil}^k \binom{k}{i}=2k\sum_{i=\lceil \frac{k}{2}\rceil -1}^{k-1} \binom{k-1}{i}-k2^{k-1}=\] \[2k\sum_{i=\lceil \frac{k-1}{2}\rceil}^{k-1} \binom{k-1}{i}-k2^{k-1}=2k\left [2^{k-2}+\frac{1}{2}\binom{k-1}{\frac{k-1}{2}}\right ]-k2^{k-1}=\] \[k\binom{k-1}{\frac{k+1}{2}-1}=\frac{k+1}{2}\binom{k}{\frac{k+1}{2}}=\left\lceil\frac{k}{2}\right\rceil\binom{k}{\lceil\frac{k}{2}\rceil};\]
if $k$ is even, then
\[S_k=2\sum_{i=\lceil \frac{k}{2}\rceil}^k i\binom{k}{i}- k\sum_{i=\lceil \frac{k}{2}\rceil}^k \binom{k}{i}=2k\sum_{i=\lceil \frac{k}{2}\rceil -1}^{k-1} \binom{k-1}{i}-k\left (2^{k-1}+\frac{1}{2}\binom{k}{k/2}\right )=\] \[2k\left ( \sum_{i=\lceil \frac{k-1}{2}\rceil}^{k-1} \binom{k-1}{i}+\binom{k-1}{\frac{k}{2}-1}\right )-k2^{k-1}-\frac{k}{2}\binom{k}{k/2}=\] \[2k2^{k-2}+2k\binom{k-1}{\frac{k}{2}-1}-k2^{k-1}-\frac{k}{2}\binom{k}{k/2}=2k\binom{k-1}{\frac{k}{2}-1}-\frac{k}{2}\binom{k}{k/2} =\] \[k\binom{k}{k/2}- \frac{k}{2}\binom{k}{k/2}=\frac{k}{2}\binom{k}{k/2}=\left\lceil\frac{k}{2}\right\rceil\binom{k}{\lceil\frac{k}{2}\rceil}.\]
It follows that \[\mu_n(A\times B)\leq \frac{1}{n2^k} \left\lceil\frac{k}{2}\right\rceil\binom{k}{\lceil\frac{k}{2}\rceil}.\]

Clearly, $-\mu_n(A\times B)=\mu_n(A'\times B)$, where $A'\subset K_n$ with $\varphi_n (A')=\{-s: s\in \varphi_n (A)\}.$ Thus \[\sup_{A\subset K_n} |\mu_n(A\times B)|\leq \frac{1}{n2^k} \left\lceil\frac{k}{2}\right\rceil\binom{k}{\lceil\frac{k}{2}\rceil}.\]
For $0\leq i\leq n$ let $B_{(n,i)}$ be the set of all $s\in K_n$ such that

$|\{j\in B: \varphi_n(s)(j)=1\}|=i$; clearly $|B_{(n,i)}|=2^{n-k}\binom{k}{i}.$

\smallskip

Put $B_{(n)}=\bigcup_{i=\lceil \frac{n}{2}\rceil}^n B_{(n,i)}.$ Then as above we get $\mu_n(B_{(n)}\times B)=$ \[\frac{1}{n2^n}\sum_{i=\lceil \frac{k}{2}\rceil}^k (2i-k) |B_{(n,i)}|=\frac{1}{n2^k}\sum_{i=\lceil \frac{k}{2}\rceil}^k (2i-k) \binom{k}{i}=\frac{1}{n2^k}\left \lceil \frac{k}{2}\right \rceil\binom{k}{\lceil \frac{k}{2}\rceil}.\]
Thus \[\sup_{A\subset K_n} |\mu_n(A\times B)|= \frac{1}{n2^k} \left\lceil\frac{k}{2}\right\rceil\binom{k}{\lceil\frac{k}{2}\rceil}=\frac{m}{n}\frac{1}{4^m}\binom{2m}{m},\; \mbox{where}\;m=\left \lceil \frac{k}{2}\right \rceil.\]

It is easy to see that the function $g: \N \to \R, g(k)=\frac{1}{2^k} \left\lceil\frac{k}{2}\right\rceil\binom{k}{\lceil\frac{k}{2}\rceil}$ is non-decreasing. Thus \[\sup_{A\times B\subset K_n\times L_n} |\mu_n(A\times B)|= \frac{1}{n2^n} \left\lceil\frac{n}{2}\right\rceil\binom{n}{\lceil\frac{n}{2}\rceil}=\frac{m}{n}\frac{1}{4^m}\binom{2m}{m},\; \mbox{where}\; m=\left \lceil \frac{n}{2}\right \rceil.\]

By Wallis's formula we get \[ W_m:= \prod_{i=1}^m \frac{(2i-1)(2i+1)}{(2i)(2i)}\to_m \frac{2}{\pi}.\] It is easy to see that
\[ \frac{2m+1}{4^{2m}}\binom{2m}{m}^2=W_m=\prod_{i=1}^m \left (1-\frac{1}{4i^2}\right )\searrow_m \frac{2}{\pi}\] and \[\frac{2m}{4^{2m}}\binom{2m}{m}^2=\frac{2m}{2m+1}W_m=\frac{1}{2}\prod_{i=1}^{m-1} \left (1+\frac{1}{4i(i+1)}\right )\nearrow_m \frac{2}{\pi}.\]
It follows that \[\frac{4^m}{\sqrt{\pi (m+1)}} < \binom{2m}{m} < \frac{4^m}{\sqrt{\pi m}}\; \mbox{for every}\; m\in \N.\]
Thus
\[\frac{m}{n}\frac{1}{\sqrt{\pi (m+1)}}< \sup_{A\subset K_n} |\mu_n(A\times B)| <\frac{m}{n}\frac{1}{\sqrt{\pi m}}\] for every $\emptyset \neq B\subset L_n$ and $m=\left \lceil\frac{|B|}{2}\right \rceil.$
Hence we get
\[\frac{1}{2\sqrt{\pi}}\frac{1}{\sqrt{n}}< \sup_{A\times B\subset K_n\times L_n} |\mu_n(A\times B)| <\frac{2}{\sqrt{\pi}}\frac{1}{\sqrt{n}}\; \mbox{for every}\; n\in \N.\]

\bigskip

 Let $f\in C(K)$ and $g\in C(L).$
 Put $$\mu_n(s,j)= \mu_n(\{(s,j)\}),\,\, (s,j) \in K_n\times L_n.$$

\bigskip

 For some $t\in \{-1,1\}^{K_n}$ and $d\in \{-1,1\}^{L_n}$ we have
\[|\mu_n(f\otimes g)|=\left |\sum_{(s,j)\in K_n\times L_n} f(s)g(j)\mu_n(s,j)\right |=\] \[\left |\sum_{s\in K_n} f(s) \sum_{j\in L_n} g(j)\mu_n(s,j)\right |\leq \sum_{s\in K_n} |f(s)| \left |\sum_{j\in L_n} g(j)\mu_n(s,j)\right |\leq\] \[ \|f\|_{\infty} \sum_{s\in K_n} \left |\sum_{j\in L_n} g(j)\mu_n(s,j)\right |= \|f\|_{\infty} \sum_{s\in K_n} t(s)\left (\sum_{j\in L_n} g(j)\mu_n(s,j)\right )=\] \[\|f\|_{\infty}\sum_{j\in L_n} g(j) \left (\sum_{s\in K_n} t(s)\mu_n(s,j)\right )\leq \] \[\|f\|_{\infty}\sum_{j\in L_n} |g(j)| \left |\sum_{s\in K_n} t(s)\mu_n(s,j)\right |\leq
\|f\|_{\infty}\|g\|_{\infty}\sum_{j\in L_n} \left |\sum_{s\in K_n} t(s)\mu_n(s,j)\right |=\] \[ \|f\|_{\infty}\|g\|_{\infty}\sum_{j\in L_n} d(j)\left ( \sum_{s\in K_n} t(s)\mu_n(s,j)\right )=\] \[\|f\|_{\infty}\|g\|_{\infty}\sum_{(s,j)\in K_n\times L_n} t(s)d(j)\mu_n(s,j).\]
Put $K_n'=\{s\in K_n: t(s)=1\}, K_n''=K_n\setminus K_n', L_n'=\{j\in L_n: d(j)=1\}$ and $L_n''=L_n\setminus L_n'$. Then \[\sum_{(s,j)\in K_n\times L_n} t(s)d(j)\mu_n(s,j)=
\sum_{(s,j)\in K_n'\times L_n'}\mu_n(s,j) +\] \[ \sum_{(s,j)\in K_n''\times L_n''} \mu_n(s,j)- \sum_{(s,j)\in K_n'\times L_n''} \mu_n(s,j) -\sum_{(s,j)\in K_n''\times L_n'} \mu_n(s,j)=\] \[ \mu_n(K_n'\times L_n')+ \mu_n(K_n''\times L_n'')- \mu_n(K_n'\times L_n'') -  \mu_n(K_n''\times L_n')\leq\] \[ 4\sup _{A\times B \subset K\times L} |\mu_n(A\times B)|\leq \frac{8}{\sqrt{\pi}}\frac{1}{\sqrt{n}}.\]
Thus \[|\mu_n(f\otimes g)|\leq \frac{8}{\sqrt{\pi}}\frac{1}{\sqrt{n}}\|f\otimes g\|_{\infty}.\] Hence \[|\mu_n(f\oplus g)|= |\mu_n(f\otimes {\bf 1}_L) + \mu_n({\bf 1}_K \otimes g)\leq \frac{8}{\sqrt{\pi}}\frac{1}{\sqrt{n}} (\|f\|_{\infty} +\|g\|_{\infty}),\] so \[|\mu_n(f\oplus g)|\leq \frac{8}{\sqrt{\pi}}\frac{1}{\sqrt{n}}\|f\oplus g\|_{\infty}.\]

\bigskip

We know that for every $A\times B \subset X\times Y$ we have
\[\mu_n(A\times B)= \mu_n ((A\cap K_n)\times (B\cap L_n)) \to_n 0.\]

Denote the family of all clopen subsets of a topological space $S$ by ${\mathcal C}(S)$. Put $K_0=\bigcup_{n=1}^{\infty} K_n$ and $L_0=\bigcup_{n=1}^{\infty}L_n$. The countable subspaces $K_0\subset K$ and $L_0 \subset L$ are strongly zero-dimensional \cite[Corollary 6.2.8]{Eng}, so $\beta K_0$ and $\beta L_0$ are zero-dimensional  \cite[Theorems 6.2.12 and 6.2.6]{Eng}. Thus for every two different points $(x_1, y_1)$ and $(x_2, y_2)$ of $Z=\beta K_0 \times \beta L_0$ there exist sets $A\in {\mathcal C}(\beta K_0)$ and $B \in {\mathcal C}(\beta L_0)$ such that $(x_1, y_1)\in A\times B$ and  $(x_2, y_2) \not\in A\times B$.

Thus the subalgebra \[{\mathcal A}:=\mbox{lin}\; \{{\bf 1}_{A\times B}: A\in {\mathcal C}(\beta K_0), B\in {\mathcal C}(\beta L_0)\}\] of $C(Z)$ separates points of $Z$. Using the Stone-Weierstrass theorem we infer that ${\mathcal A}$ is dense in $C(Z)$. It follows that $\mu_n(f) \to_n 0$ for every $f\in C(Z)$, since $\sup_n \|\mu_n\|<\infty$ and, as one can easily see, $\mu_n(f) \to_n 0$ for every $f\in {\mathcal A}$. By \cite[Corollary 3.6.6]{Eng} there exist continuous maps \[G: \beta K_0 \to K \;\mbox{and}\; H: \beta L_0 \to L\] such that $G(x)=x$ and $H(y)=y$ for all $x\in K_0$ and $y\in L_0$.

For every $F\in C(X\times Y)$ the function \[f: Z \to \R, z=(x,y) \to f(z)=F(G(x),H(y))\] is continuous and $F(x,y)=f(x,y)$ for all $(x,y)\in K_0\times L_0.$ Hence $\mu_n(F)=\mu_n(f)$ for $n\in \N$, since $\mbox{supp}\;(\mu_n) \subset K_0\times L_0.$ Thus $\mu_n(F) \to_n 0$ for every $F\in C(X\times Y).$
\end{proof}

\begin{corollary} Let $X$ and $Y$ be Tychonoff spaces that contain infinite compact subspaces $K$ and $L$, respectively. Let $(K_n\times L_n)$ be a sequence of finite subsets of $K\times L$, such that $|K_n|\to_n \infty$ and $|L_n|\to_n \infty$. Then there exists a sequence $(\mu_n)$ of normalized finitely supported signed measures on $X\times Y$ such that

(1) $\mbox{supp}\, (\mu_n)\subset K_n \times L_n$ for $n\in \N;$

(2) $\mu_n(A\times B)\to_n 0$ for every $A\times B\subset X\times Y$;

(3) $\mu_n(f) \to_n 0$ for every $f\in C(X \times Y).$
\end{corollary}
\begin{proof}
For every $m\in \N$ there exists $\phi_m\in \N$ such that $|K_n|\geq 2^m$ and $|L_n|\geq m$ for $n\geq \phi_m.$  Without loss of generality we can assume that the sequence $(\phi_m)$ is strictly increasing and $\phi_1>1.$ Let $K_n'\times L_n' \subset K_n\times L_n$ such that $|K_n'|=1, |L_n'|=1$ if $1\leq n< \phi_1$ and $|K_n'|=2^m, |L_n'|=m$ if $\phi_m\leq n<\phi_{m+1}, m\in \N.$

For $1\leq n< \phi_1$ we put $\mu_n = \delta_{(x_n,y_n)},$ where $\{(x_n,y_n)\}=K_n'\times L_n'.$

Let $\varphi_n: K_n' \to 2^{L_n'}$ be a bijection for $n\geq \phi_1.$ Let $$\mu_n= \sum_{(s,j)\in K_n'\times L_n'} \varphi_n(s)(j) \delta_{(s,j)}$$ for $n\geq \phi_1.$ Clearly $\|\mu_n\|=1$ and $\mbox{supp}\; (\mu_n)=K_n'\times L_n' \subset K_n\times L_n$ for $n\in \N.$

By Theorem 1 and its proof we infer that for every $A\times B \subset X\times Y$ we have $$|\mu_n (A\times B)| \leq \frac{2}{\sqrt{\pi}}\frac{1}{\sqrt{m}}$$ for $\phi_m\leq n<\phi_{m+1}, m\in \N$, so $\mu_n (A\times B) \to_n 0$ and $\mu_n (f) \to_n 0$ for every $f\in C(X\times Y).$
 \end{proof}

 By the proof of Theorem \ref{MAIN} we get also the following corollary.

\begin{corollary} Let $(\mu_n)$ be the sequence from Theorem 1. Then  \[\sup_{A\subset K_n} |\mu_n(A\times B)|=\frac{1}{n2^k}\sum_{i=\lceil \frac{k}{2}\rceil}^k (2i-k) \binom{k}{i}= \frac{1}{n2^k} \left\lceil\frac{k}{2}\right\rceil\binom{k}{\lceil\frac{k}{2}\rceil}=\frac{m}{n}\frac{1}{4^m}\binom{2m}{m}\] and \[\frac{m}{n}\frac{1}{\sqrt{\pi (m+1)}}< \sup_{A\subset K_n} |\mu_n(A\times B)| <\frac{m}{n}\frac{1}{\sqrt{\pi m}}\] for every $\emptyset \neq B\subset L_n$ and $k=|B|, m=\left \lceil\frac{k}{2}\right \rceil, n\in \N.$
\end{corollary}

Using Theorem 1 we obtain the following.

\begin{theorem} Let $K$ and $L$ be infinite compact spaces and let $(\mu_n)$ be a JN-sequence from Theorem 1. Then every subsequence $(\mu_{k_n})$  of $(\mu_n)$ contains a strongly normal sebsequence that is an $\omega^*$-basic sequence in the dual of the Banach space $C(K\times L)$. \end{theorem}

\begin{proof} For every two different points $(x_1, y_1)$ and $(x_2, y_2)$ of $K \times L$ we have $x_1\neq x_2$ or $y_1\neq y_2$. Thus there exists $f\in C(K)$ with $f(x_1)=1$ and $f(x_2)=0$ or $g\in C(L)$ with $g(y_1)=1$ and $g(y_2)=0$, so $f\otimes {\bf 1}_L$ or ${\bf 1}_K \otimes g$ separates points $(x_1, y_1)$ and $(x_2, y_2)$.

Thus the subalgebra \[{\mathcal A}:=\mbox{lin}\; \{f\otimes g: (f, g)\in C(K)\times C(L)\} \] of $C(K \times L)$ separates points of $K \times L$. Using the Stone-Weierstrass theorem we infer that ${\mathcal A}$ is dense in $C(K \times L)$.

Let $(s_{k_n})$ be a subsequence of $(k_n)$ such that $\sum_{n=1}^{\infty} \frac{1}{\sqrt{s_{k_n}}}<\infty.$ Then for all $(f, g)\in C(K)\times C(L)$ we have \[ \sum_{n=1}^{\infty} |\mu_{s_{k_n}}(f\otimes g)|\leq \sum_{n=1}^{\infty} \frac{8}{\sqrt{\pi}}\frac{1}{\sqrt{s_{k_n}}}\|f\otimes g\|_{\infty}<\infty.\] It follows that $\sum_{n=1}^{\infty} |\mu_{s_{k_n}}(h)|< \infty$ for every $h\in {\mathcal A}$. Thus the sequence $(\mu_{s_{k_n}})$ is strongly normal. By \cite[Theorem 1]{WS}, every strongly normal sequence $(y_n)$ in the dual of a Banach space $E$ contains a subsequence that is an $\omega^*$-basic sequence in the dual of $E$. Using this theorem we complete the proof. \end{proof}

For a Tychonoff space $X$ by $C_k(X)$ we denote the space $C(X)$ endowed with the compact-open topology. It is well known (see \cite{mccoy}) that $C_k(X)$ is a Fr\'echet lcs, i.e. a metrizable and complete lcs, if and only if $X$ is a hemicompact $k_{R}$-complete space. Theorem \ref{MAIN} applies easily  to get also the following  \cite[Corollary 6.7]{BKS}.
\begin{corollary} [Bargetz-K\c akol-Sobota]\label{BKS}
Let $X$  and $Y$ be locally compact and $\sigma$-compact spaces. Then the Fr\'echet lcs $C_k(X\times Y)$   contains either a complemented copy of $\mathbb{R}^{\mathbb{N}}$ or a complemented copy of $c_0$. Consequently, if both spaces $X$ and $Y$ are uncountable, $C_k(X\times Y)$ contains a complemented copy of $c_0$.
\end{corollary}
\begin{proof}
First assume that both spaces $X$ and $Y$ contain infinite compact subsets.  Then by Theorem \ref{MAIN} the space $C_p(X\times Y)$ has a JN-sequence. We apply Theorem \ref{BKS} to deduce that $C_p(X\times Y)$ contains a complemented copy of $(c_0)_{p}$. The classical closed graph theorem between Fr\'echet lcs, \cite[Theorem 4.1.10]{Bonet}, applies to get that the Fr\'echet lcs   $C_k(X\times Y)$ contains a complemented copy of the Banach space $c_0$. Now assume that, for example,  $X$ contains no infinite compact subset. Then $X$ is countable and locally compact, hence must be discrete and   homeomorphic to $\mathbb{N}$. Thus $C_p(X\times Y)$ contains a complemented copy of $C_p(\mathbb{N})$ (which is isomorphic to $\mathbb{R}^{\mathbb{N}}$). Applying again the same  closed graph theorem between Fr\'echet lcs, the Fr\'echet lcs $C_k(X\times Y)$ contains a complemented copy of the Fr\'echet lcs  $\mathbb{R}^{\mathbb{N}}$.
\end{proof}
Next example shows that the assumptions on $X$ and $Y$ in Theorem \ref{MAIN} cannot be omitted, compare with Corollary \ref{BKS}.
\begin{example}\label{example1}
The product $\mathbb{N}\times \beta\mathbb{N}$ does not have a JN-sequence.
\end{example}
\begin{proof}
With Theorem \ref{BKS} in mind, let us assume that  $C_p(\mathbb{N}\times\beta\mathbb{N}) \simeq C_p(\beta\mathbb{N})^{\mathbb{N}}$ contains a complemented copy of $(c_0)_p$. The classical  closed graph theorem between Fr\'echet lcs implies that $C(\beta{N})^{\mathbb{N}}\simeq \ell_{\infty}(\mathbb{N}\times\mathbb{N})\simeq \ell_{\infty}(\mathbb{N})$ contains a complemented copy of the Banach space $c_0$, which is not true.
\end{proof}
Note also that last Example \ref{example1} can be also deduced from \cite[Theorem 3.15]{Filipov}.
We have also the following \cite[Example 3.4]{KMSZ}, \cite[Theorem 1.5]{KMSZ}.
\begin{example}[K\c akol-Marciszewski-Sobota-Zdomskyy] \label{example2}
There exists a pseudocompact space $H$ such that all compact subsets of $H$ are finite and $H\times H$ is pseudocompact and yet has a JN-sequence.  It is consistent that there exists an infinite pseudocompact space $X$ such
$X\times X$  does  not admit a JN-sequence.
\end{example}
\section{Complemented JN-sequences} 
We provide some  application of our Theorem \ref{MAIN}.
Next useful Proposition \ref{LAST} describes another copies  of  $(c_0)_p$ complemented in $C_p(X)$ for special (but natural) JN-sequences.
\begin{proposition}\label{LAST}  Let $X$ be a Tychonoff space with a JN-sequence $(\mu_n)$ such that the supports of $\mu_n, n\in \mathbb{N},$ are pairwise disjoint and their sum is a discrete subset of $X$. Then there exists a sequence of functions $(\varphi_n) \subset C(X, [0,1])$ with pairwise disjoint supports such that $\mu_n (\varphi_m)=0$ for all $n,m \in \mathbb{N}, n\neq m$ and $\inf_n |\mu_n (\varphi_n)|>0.$ It follows that $E= \{\sum_{n=1}^{\infty} x_n \varphi_n: (x_n) \in c_0\}$ is a complemented subspace of $C_p(X)$ isomorphic to $(c_0)_p$.
 \end{proposition}
\begin{proof} (1) Let $A_{2n}\subset \mbox{supp}\, \mu_n$ with $|\mu_n (A_{2n})|\geq \frac{1}{2}$ and $A_{2n-1}= (\mbox{supp}\, \mu_n \setminus A_{2n})$ for $n\in \mathbb{N}.$ Put $$A:=\bigcup_{n=1}^{\infty} A_n=\{x_n: n\in \mathbb{N}\}.$$ Denote the topology of $X$ by $\tau$. The subset $A$ of $X$ is discrete, so there exists a sequence $(U_n) \subset \tau$ such that $U_n\cap A=\{x_n\}, n\in \mathbb{N}.$

Let $(W_n) \subset \tau$ such that $$x_n \in W_n \subset \overline{W_n} \subset U_n, n\in \mathbb{N}.$$

The open sets $$V_n:=W_n \setminus \bigcup_{1\leq  k<n} \overline{W_k}, n\in \mathbb{N},$$ are pairwise disjoint and $x_n \in V_n, n\in \mathbb{N}.$ Let $(B_n) \subset \tau$ such that $x_n \in B_n \subset \overline{B_n} \subset V_n, n\in \mathbb{N}.$  Put $I_n:= \{k\in \mathbb{N}: x_k\in A_n\}, C_n:=\bigcup_{k\in I_n} B_k$ and $D_n:=\bigcup_{k\in I_n} V_k$ for $n\in \mathbb{N}.$ Then $(C_n), (D_n) \subset \tau$ and $A_n \subset C_n \subset \overline{C_n} \subset D_n$ for $n\in \mathbb{N}.$ Clearly, the sets $D_n, n\in \mathbb{N},$ are pairwise disjoint.

For every $n\in \mathbb{N}$ there exists a continuous function $f_n: X \to [0,1]$ such that $f_n(x_n)=1$ and $f_n|B_n^c=0$. The functions $g_n:=\max_{k\in I_n} f_k, n\in \mathbb{N},$ are continuous, $g_n(X)\subset [0,1]$ and $\mbox{supp}\, g_n\subset D_n$ for $n\in \mathbb{N}.$ Put $\varphi_n=g_{2n}$ for  $n\in \mathbb{N}.$ Since $$\mbox{supp}\, \mu_n=A_{2n}\cup A_{2n-1}\subset D_{2n} \cup D_{2n-1}$$ and $$\mbox{supp}\, \varphi_n \subset D_{2n}$$ for $n\in \mathbb{N},$ we get $\mu_n (\varphi_m)=0$ for all $n,m \in \N, n\neq m.$ $\\$ Moreover $\inf_n |\mu_n (\varphi_n)|=\inf_n |\mu_n (A_{2n})|\geq \frac{1}{2}.$

\bigskip

(2) Put $t_n:=1/\mu_n(\varphi_n), n\in \mathbb{N}$. The operator \[ T: (c_0)_p \to C_p(X), x=(x_n) \to Tx=\sum_{n=1}^{\infty} t_n x_n \varphi_n\] is well defined, linear, injective and continuous, since the functions $\varphi_n, n\in \mathbb{N},$ have pairwise disjoint supports, $(|t_n|) \subset [1, 2]$ and $\varphi_n (X) \subset [0,1], n\in \mathbb{N}.$ The linear operator \[S: C_p(X) \to (c_0)_p, f \to Sf=(\mu_n (f))\] is well defined and continuous. For $x=(x_k) \in c_0$ and $n\in \mathbb{N}$ we have \[\mu_n (Tx)=\sum_{k=1}^{\infty} t_kx_k \mu_n (\varphi_k)=t_nx_n\mu_n(\varphi_n)=x_n.\] Thus $STx=x$ for every $x\in c_0.$ Hence the operator \[P: C_p(X) \to C_p(X), P=TS,\] is a linear continuous projection. Thus the subspace $Z:= \ker P= \ker S$ is complemented in $C_p(X)$. The operator $S$ is open, since for every neighbourhood $U$ of $0$ in $C_p(X)$ the set $V:=T^{-1}(U)$ is a neighbourhood of $0$ in $(c_0)_p$, and $V=ST(V)\subset S(U).$ Thus the quotient space $C_p(X)/Z$ is isomorphic to $(c_0)_p$, and $T(c_0)=P(C_p(X))$ is a complemented subspace of $C_p(X)$, that is isomorphic to $(c_0)_p$. $\\$ Clearly $T(c_0)=\{\sum_{n=1}^{\infty} x_n \varphi_n: (x_n) \in c_0\}.$
\end{proof}
A JN-sequence $(\mu_n)$ on a Tychonoff space $X$ is said to be {\it complemented} if there exists a sequence $(\varphi_n) \subset C(X)$ of non-negative functions with pairwise disjoint supports and $\sup_{x\in X}\varphi_n(x)=1, n\in \N,$ such that $\mu_n (\varphi_m)=0$ for all $n,m \in \mathbb{N}, n\neq m$ and $\inf_n |\mu_n (\varphi_n)|>0.$
\begin{corollary}\label{LAST1}
Let $X$ and $Y$ be Tychonoff spaces that contain infinite compact subspaces. 
Then the product space $X\times Y$ has a complemented JN-sequence $(\mu_n).$ In particular, $C_p(X\times Y)$ contains a complemented subspace as provided in Proposition \ref{LAST} which is  isomorphic to $(c_0)_p$.
 \end{corollary}
\begin{proof} Let $K$ and $L$ be infinite compact subspaces of $X$ and $Y$, respectively. Let $K_0$ and $L_0$ be infinite countable discrete subsets of $K$ and $L$, respectively. Let $(K_n)$ and $(L_n)$ be partitions of $K_0$ and $L_0$, respectively, such that $|K_n|=2^n$ and $|L_n|=n$ for $n\in \mathbb{N}$. By Theorem \ref{MAIN} there exists a JN-sequence $(\mu_n)$ such that the supports of $\mu_n, n\in\mathbb{N},$ are pairwise disjoint and their sum is a discrete subset of $X\times Y$. Then Proposition \ref{LAST} completes the proof. \end{proof}
Corollary \ref{LAST1} combined with the closed graph theorem  yields  Corollary \ref{LAST2}  which   extends Cembranos-Freniche result mentioned above.
\begin{corollary}\label{LAST2}
Let $X$ and $Y$ be infinite compact spaces. Then there exists a normalized sequence $(\varphi_n)$ in the Banach space  $C(X\times Y)$ of non-negative functions with pairwise disjoint supports such that the subspace $E:=\{\sum_{n=1}^{\infty} x_n \varphi_n: (x_n) \in c_0\}$ of $C(X\times Y)$ is an isometric copy of $c_0$ and complemented in $C(X\times Y)$.
\end{corollary}
The above Corollary \ref{LAST1}  may suggest a question whether every Tychonoff space $X$ with a JN-sequence admits also  a complemented JN-sequence. It turns out that the following general fact (kindly suggested to the authors by Sobota \cite{Sobota}) also holds.
\begin{proposition}\label{pro}
For every Tychonoff space $X$ with a JN-sequence there exists in $X$  a complemented JN-sequence.
\end{proposition}
\begin{proof} By \cite[Theorem 1.2]{MSZ} there exists
 a JN-sequence $(\mu_n)$ with disjoint supports for which we can find  disjoint open sets $(U_n)$  such that   $\supp(\mu_n)\subset  U_n$ for every $n\in\mathbb{N}$. Define $\varphi_n(x) = sgn(\mu_n(\{x\}))$ for $x\in\supp(\mu_n)$
and $\varphi_n = 0$ outside $U_n.$ Then $0\leq \varphi_n\leq 1$ and   $\mu_n(\varphi_n)=1$ but
$\mu_n(\varphi_m)=0$ for every $m\neq n$,  so  $(\mu_n)$ is a complemented
JN-sequence.
\end{proof}



\end{document}